\documentclass[12pt]{amsart}
\usepackage{mathptmx}
\usepackage{amssymb}
\usepackage{amsfonts}
\usepackage{amsmath}
\usepackage{graphicx}
\usepackage[pagebackref]{hyperref}

\theoremstyle{plain}
\newtheorem{thm}{Theorem}[section]

\newtheorem{lemma}[thm]{Lemma}
\newtheorem{proposition}[thm]{Proposition}

\theoremstyle{definition}
\newtheorem{definition}[thm]{Definition}

\theoremstyle{remark}
\newtheorem{remark}[thm]{Remark}
\newtheorem{example}[thm]{Example}

\DeclareMathOperator{\htt}{ht}
\DeclareMathOperator{\qf}{qf}
\DeclareMathOperator{\Max}{Max}
\DeclareMathOperator{\Cl}{Cl}
\DeclareMathOperator{\Pic}{Pic}

\makeatletter
  \newcounter{xenumi}
  \newenvironment{xenumerate}{%
  \begin{list}{(\arabic{xenumi})}{
    \setcounter{xenumi}{1}\usecounter{xenumi}
    \setlength{\parsep}{4\p@ \@plus2\p@ \@minus\p@}
    \setlength{\topsep}{6\p@ \@plus2\p@ \@minus2\p@}
    \setlength{\itemsep}{2\p@ \@plus1\p@ \@minus\p@}
    \setlength{\labelwidth}{0mm}
    \setlength{\labelsep}{2mm}
    \setlength{\itemindent}{2mm}
    \setlength{\leftmargin}{0mm}
    \setlength{\listparindent}{0mm}
  }}{\end{list}}
\makeatother

\begin{document}

\title[\boldmath$t$-Class semigroups of Noetherian domains]{\boldmath$t$-Class semigroups of Noetherian domains}

\author{S. Kabbaj}
\address{Department of Mathematical Sciences, King Fahd University of Petroleum \& Minerals, P.O. Box 5046, Dhahran 31261, KSA}
\email{kabbaj@kfupm.edu.sa}

\author{A. Mimouni}
\address{Department of Mathematical Sciences, King Fahd University of Petroleum \& Minerals, P.O. Box 5046, Dhahran 31261, KSA}
\email{amimouni@kfupm.edu.sa}
\thanks{This work was funded by KFUPM under Project \# MS/t-Class/257.}

\date{\bf \today}

\begin{abstract}
The $t$-class semigroup of an integral domain $R$, denoted ${\mathcal S}_{t}(R)$, is the semigroup of fractional $t$-ideals modulo its subsemigroup of nonzero principal ideals with the operation induced by ideal $t$-multiplication. This paper investigates ring-theoretic properties of a Noetherian domain that reflect reciprocally in the Clifford or Boolean property of its $t$-class semigroup.
\end{abstract}

\maketitle

\begin{section}{Introduction}

Let $R$ be an integral domain. The class semigroup of $R$, denoted
${\mathcal S}(R)$, is the semigroup of nonzero fractional ideals
modulo its subsemigroup of nonzero principal ideals \cite{BS:1996}, \cite{ZZ:1994}. We
define the $t$-class semigroup of $R$, denoted ${\mathcal
S}_{t}(R)$, to be the semigroup of fractional $t$-ideals modulo its
subsemigroup of nonzero principal ideals, that is, the semigroup of
the isomorphy classes of the $t$-ideals of $R$ with the operation
induced by $t$-multiplication. Notice that ${\mathcal S}_{t}(R)$ stands
as the $t$-analogue of ${\mathcal S}(R)$, whereas the class
group $\Cl(R)$ is the $t$-analogue of the Picard group $\Pic(R)$. In general, we have
 $$\Pic(R)\subseteq \Cl(R)\subseteq {\mathcal{S}_{t}}(R)\subseteq {\mathcal S}(R)$$
where the first and third containments turn into equality if $R$ is a Pr\"ufer domain and
the second does so if $R$ is a Krull domain.

A commutative semigroup $S$ is said to be Clifford if every element
$x$ of $S$ is (von Neumann) regular, i.e., there exists $a\in S$
such that $x=ax^{2}$. A Clifford semigroup $S$ has the ability to stand as a disjoint union of subgroups
$G_e$, where $e$ ranges over the set of idempotent elements of $S$,
and $G_e$ is the largest subgroup of $S$ with identity equal to $e$
(cf. \cite{Ho:1995}). The semigroup $S$ is said to be Boolean if for each
$x\in S$, $x=x^{2}$. A domain $R$ is said to be \emph{Clifford}
(resp., \emph{Boole}) \emph{$t$-regular} if $S_{t}(R)$ is a Clifford
(resp., Boolean) semigroup.

This paper investigates the $t$-class semigroups of Noetherian domains. Precisely, we study conditions under which
$t$-stability characterizes $t$-regularity. Our first result, Theorem~\ref{sec:2.2}, compares Clifford $t$-regularity to various forms of stability. Unlike regularity, Clifford (or even Boole) $t$-regularity over Noetherian domains does not force the $t$-dimension to be one (Example~\ref{sec:2.4}). However, Noetherian strong $t$-stable domains happen to have $t$-dimension 1. Indeed, the main result, Theorem~\ref{sec:2.6}, asserts that ``\emph{$R$ is strongly $t$-stable if and only if $R$ is Boole $t$-regular and $t$-$\dim(R)=1$}.'' This result is not valid for Clifford $t$-regularity as shown by Example~\ref{sec:2.9}. We however extend this result to the Noetherian-like larger class of strong Mori domains (Theorem~\ref{sec:2.10}).

All rings considered in this paper are integral domains. Throughout, we shall use $\qf(R)$ to denote the quotient
field of a domain $R$, $\overline{I}$ to denote the isomorphy class of a $t$-ideal $I$ of $R$ in $S_{t}(R)$, and $\Max_{t}(R)$ to denote the set of maximal $t$-ideals of $R$.
\end{section}

\begin{section}{Main results}\label{sec:2}

We recall that for a nonzero fractional ideal $I$ of $R$,
$I_{v}:=(I^{-1})^{-1}$, $I_{t}:=\bigcup J_{v}$ where $J$
ranges over the set of finitely generated subideals of $I$, and
$I_{w}:=\bigcup (I:J)$ where the union is taken over all
finitely generated ideals $J$ of $R$ with $J^{-1}=R$.  The ideal $I$ is said to be
divisorial or a $v$-ideal if $I=I_{v}$, a $t$-ideal if $I=I_{t}$,
and a $w$-ideal if $I=I_{w}$. A domain $R$ is called \emph{strong Mori}
if $R$ satisfies the ascending chain condition on $w$-ideals \cite{FM:1999}. Trivially, a Noetherian domain is strong Mori and a strong
Mori domain is Mori. Suitable background on strong Mori domains is \cite{FM:1999}. Finally, recall that the \emph{$t$-dimension} of $R$, abbreviated
$t$-$\dim(R)$, is by definition equal to the length of the longest chain of $t$-prime ideals of $R$.

The following lemma displays necessary and sufficient conditions for $t$-regularity. We often will be appealing to this lemma without explicit mention.

\begin{lemma}[{\cite[Lemma 2.1]{KM2:2007}}]\label{sec:2.1}
Let $R$ be a domain.
\begin{enumerate}
\item $R$ is Clifford $t$-regular if and only if, for each $t$-ideal $I$ of $R$, $I=(I^{2}(I:I^{2}))_{t}$.
\item $R$ is Boole $t$-regular if and only if, for each $t$-ideal $I$ of $R$, $I=c(I^{2})_{t}$ for some $c\not=0\in \qf(R)$.\qed
\end{enumerate}
\end{lemma}

An ideal $I$ of a domain $R$ is said to be \emph{$L$-stable}
(here $L$ stands for Lipman) if $R^{I}:=\bigcup_{n\geq1}(I^{n}:I^{n})=(I:I)$,
and $R$ is called $L$-stable if every nonzero ideal is
$L$-stable. Lipman introduced the notion of stability in the specific setting of
one-dimensional commutative semi-local Noetherian rings in order to
give a characterization of Arf rings; in this context, L-stability
coincides with Boole regularity \cite{Lip:1971}.

Next, we state our first theorem of this section.

\begin{thm}\label{sec:2.2}
Let $R$ be a Noetherian domain and consider the following statements:
\begin{enumerate}
\item $R$ is Clifford $t$-regular;
\item Each $t$-ideal $I$ of $R$ is $t$-invertible in $(I:I)$;
\item Each $t$-ideal is L-stable.
\end{enumerate}
Then $(1) \Longrightarrow (2) \Longrightarrow (3)$. Moreover, if $t$-$\dim(R)=1$, then $(3) \Longrightarrow (1)$.
\end{thm}

\begin{proof}
$(1)\Longrightarrow (2)$. Let $I$ be a
$t$-ideal of a domain $A$. Then for each ideal $J$ of $A$, $(I:J) =
(I:J_{t})$. Indeed, since $J\subseteq J_{t}$, then $(I:J_{t})\subseteq
(I:J)$. Conversely, let $x\in (I:J)$. Then $xJ\subseteq I$ implies
that $xJ_{t}=(xJ)_{t}\subseteq I_{t}=I$, as claimed. So $x\in (I:J_{t})$ and
therefore $(I:J)\subseteq (I:J_{t})$. Now, let $I$ be a $t$-ideal of
$R$, $B=(I:I)$ and $J=I(B:I)$. Since ${\overline I}$ is regular in
${\mathcal S_t}(R)$, then $I = (I^{2}(I:I^{2}))_{t}= (IJ)_{t}$. By
the claim, $B=(I:I)=(I:(IJ)_{t})=(I:IJ)= ((I:I):J)=(B:J)$. Since $B$
is Noetherian, then $(I(B:I))_{t_{1}}=J_{t_{1}}=J_{v_{1}} =B$, where
$t_{1}$- and $v_{1}$ denote the $t$- and $v$-operations with respect
to $B$. Hence $I$ is $t$-invertible as an ideal of $(I:I)$.

$(2)\Longrightarrow (3)$. Let $n\geq 1$, and $x\in (I^{n}:I^{n})$.
Then $xI^{n}\subseteq I^{n}$ implies that $xI^{n}(B:I)\subseteq
I^{n}(B:I)$. So $x(I^{n-1})_{t_{1}}=x(I^{n}(B:I))_{t_{1}}\subseteq
(I^{n}(B:I))_{t_{1}}= (I^{n-1})_{t_{1}}$. Now, we iterate this
process by composing the two sides by $(B:I)$, applying the
$t$-operation with respect to $B$ and using the fact that $I$ is
$t$-invertible in $B$, we obtain that $x\in (I:I)$. Hence $I$ is
L-stable.

$(3)\Longrightarrow (1)$ Assume that $t$-$\dim(R)=1$. Let $I$ be a
$t$-ideal of $R$ and $J = (I^{2}(I:I^{2}))_{t}=
(I^{2}(I:I^{2}))_{v}$ (since $R$ is Noetherian, and so a
$TV$-domain). We wish to show that $I=J$. By \cite[Proposition
2.8.(3)]{Kg:1987}, it suffices to show that $IR_{M}=JR_{M}$ for each
$t$-maximal ideal $M$ of $R$. Let $M$ be a $t$-maximal ideal of $R$.
If $I\not\subseteq M$, then $J\not\subseteq M$. So $IR_{M} = JR_{M}
= R_{M}$. Assume that $I\subseteq M$. Since $t$-$\dim(R) =1$, then
$\dim(R)_{M}= 1$. Since $IR_{M}$ is L-stable, then by
\cite[Lemma 1.11]{Lip:1971} there exists a nonzero element $x$ of $R_{M}$
such that $I^{2}R_{M} = xIR_{M}$. Hence $(IR_{M}:I^{2}R_{M})=
(IR_{M}:xIR_{M}) = x^{-1}(IR_{M}:IR_{M})$. So
$I^{2}R_{M}(IR_{M}:I^{2}R_{M})=xIR_{M}x^{-1}(IR_{M}:IR_{M})=IR_{M}$.
Now, by \cite[Lemma 5.11]{Kg:1987},
$JR_{M}=((I^{2}(I:I^{2}))_{v})R_{M}=(I^{2}(I:I^{2}))R_{M})_{v}=
(I^{2}R_{M}(IR_{M}:I^{2}R_{M}))_{v}=(IR_{M})_{v}
=I_{v}R_{M}=I_{t}R_{M}=IR_{M}$.
\end{proof}

According to \cite[Theorem 2.1]{Ba4:2001} or \cite[Corollary 4.3]{KM:2003}, a
Noetherian domain $R$ is Clifford regular if and only if $R$ is
stable if and only if $R$ is $L$-stable and $\dim(R)=1$. Unlike
Clifford regularity, Clifford (or even Boole) $t$-regularity does
not force a Noetherian domain $R$ to be of $t$-dimension one. In order to illustrate this fact, we first establish the transfer of Boole $t$-regularity to pullbacks issued from local Noetherian domains.

\begin{proposition}\label{sec:2.3}
Let $(T, M)$ be a local Noetherian domain with residue field $K$ and $\phi:T\longrightarrow K$ the canonical
surjection. Let $k$ be a proper subfield of $K$ and $R:=\phi^{-1}(k)$ the pullback issued from the following diagram of canonical
homomorphisms:
\[\begin{array}{ccl}
R            & \longrightarrow                 & k\\
\downarrow   &                                 & \downarrow\\
T            & \stackrel{\phi}\longrightarrow  & K=T/M
\end{array}\]
Then $R$ is Boole $t$-regular if and only if so is $T$.
\end{proposition}

\begin{proof}
By \cite[Theorem 4]{BR:1976} (or \cite [Theorem 4.12]{GH:1997}) $R$ is
a Noetherian local domain with maximal ideal $M$. Assume that $R$ is
Boole $t$-regular. Let $J$ be a $t$-ideal of $T$. If $J(T:J)=T$,
then $J=aT$ for some $a\in J$ (since $T$ is local). Then $J^{2}=aJ$
and so $(J^{2})_{t_{1}}=aJ$, where $t_{1}$ is the $t$-operation
with respect to $T$ (note that $t_{1}=v_{1}$ since $T$ is
Noetherian), as desired. Assume that $J(T:J)\subsetneq T$.  Since $T$
is local with maximal ideal $M$, then $J(T:J)\subseteq M$. Hence
$J^{-1}=(R:J)\subseteq (T:J)\subseteq (M:J)\subseteq J^{-1}$ and
therefore $J^{-1}=(T:J)$. So
$(T:J^{2})=((T:J):J)=((R:J):J)=(R:J^{2})$. Now, since $R$ is Boole
$t$-regular, then there exists $0\not =c\in \qf(R)$ such that
$(J^{2})_{t}=((J_{t})^{2})_{t}=cJ_{t}$. Then
$(T:J^{2})=(R:J^{2})=(R:(J^{2})_{t})=(R:cJ_{t})=c^{-1}(R:J_{t})=c^{-1}(R:J)=c^{-1}(T:J)$.
Hence $(J^{2})_{t_{1}}=(J^{2})_{v_{1}}=cJ_{v_{1}}=cJ_{t_{1}}=cJ$, as
desired. It follows that $T$ is Boole $t$-regular.

Conversely, assume that $T$ is Boole $t$-regular and let $I$ be a
$t$-ideal of $R$. If $II^{-1}=R$, then $I=aR$ for some $a\in I$. So
$I^{2}=aI$, as desired. Assume that $II^{-1}\subsetneq R$. Then
$II^{-1}\subseteq M$. So $T\subseteq (M:M)=M^{-1}\subseteq
(II^{-1})^{-1}= (I_{v}:I_{v})=(I:I)$. Hence $I$ is an ideal of $T$.
If $I(T:I)=T$, then $I=aT$ for some $a\in I$ and so $I^{2}=aI$, as
desired. Assume that $I(T:I)\subsetneq T$. Then $I(T:I)\subseteq M$,
and so $I^{-1}\subseteq (T:I)\subseteq (M:I)\subseteq I^{-1}$. Hence
$I^{-1}=(T:I)$. So $(T:I^{2})=((T:I):I)=((R:I):I)=(R:I^{2})$. But
since $T$ is Boole $t$-regular, then there exists $0\not =c\in
qf(T)=\qf(R)$ such that
$(I^{2})_{t_{1}}=((I_{t_{1}})^{2})_{t_{1}}=cI_{t_{1}}$. Then
$(R:I^{2})=(T:I^{2})=(T:(I^{2})_{t_{1}})=(T:cI_{t_{1}})=c^{-1}(T:I_{t_{1}})=c^{-1}(T:I)=c^{-1}(R:I)$.
Hence $(I^{2})_{t}=(I^{2})_{v}=cI_{v}=cI_{t}=cI$, as
desired. It follows that $R$ is Boole $t$-regular.
\end{proof}

Now we are able to build an example of a Boole $t$-regular Noetherian domain with $t$-dimension $\gneqq 1$.

\begin{example}\label{sec:2.4}
Let $K$ be a field, $X$ and $Y$ two indeterminates over $K$, and $k$ a
proper subfield of $K$. Let $T:=K[[X,Y]]=K+M$ and $R:=k+M$ where $M:=(X,Y)$. Since $T$ is a
UFD, then $T$ is Boole $t$-regular \cite[Proposition 2.2]{KM2:2007}. Further, $R$ is a Boole $t$-regular
Noetherian domain by Proposition~\ref{sec:2.3}. Now $M$ is a $v$-ideal of $R$, so that $t$-$\dim(R)=\dim(R)=2$.
\end{example}

Recall that an ideal $I$ of a domain $R$ is said to be \emph{stable} (resp.,
\emph{strongly stable}) if $I$ is invertible (resp., principal) in its
endomorphism ring $(I:I)$, and $R$ is called a stable (resp.,
strongly stable) domain provided each nonzero ideal of $R$ is stable
(resp., strongly stable). Sally and Vasconcelos \cite{SV:1973} used this
concept to settle Bass'conjecture on one-dimensional Noetherian
rings with finite integral closure. Recall that a stable domain is
$L$-stable \cite[Lemma 2.1]{AHP:1987}. For recent developments on
stability, we refer the reader to \cite{AHP:1987} and \cite{O1:1998, O2:2001, O3:2002}. By analogy, we define the following concepts:

\begin{definition}\label{sec:2.5}
A domain $R$ is \emph{$t$-stable} if each $t$-ideal of $R$ is stable, and $R$ is \emph{strongly $t$-stable} if each $t$-ideal of $R$ is strongly stable.
\end{definition}

Strong $t$-stability is a natural stability condition that best suits Boolean $t$-regularity. Our next theorem is a
satisfactory $t$-analogue for Boolean regularity \cite[Theorem 4.2]{KM:2003}.

\begin{thm}\label{sec:2.6}
Let $R$ be a Noetherian domain. The following conditions are equivalent:
\begin{enumerate}
\item $R$ is strongly $t$-stable;
\item $R$ is Boole $t$-regular and $t$-$\dim(R)=1$.
\end{enumerate}
\end{thm}

The proof relies on the following lemmas.

\begin{lemma}\label{sec:2.7}
Let $R$ be a $t$-stable Noetherian domain. Then $t$-$\dim(R)=1$.
\end{lemma}

\begin{proof}
Assume $t$-$\dim(R)\geq 2$. Let $(0)\subset P_{1}\subset P_{2}$ be a
chain of $t$-prime ideals of $R$ and $T:=(P_{2}:P_{2})$. Since $R$
is Noetherian, then so is $T$ (as $(R:T)\not=0$) and $T\subseteq
\overline{R}=R'$, where $\overline{R}$ and $R'$ denote respectively
the complete integral closure and the integral closure of $R$. Let
$Q$ be any minimal prime over $P_{2}$ in $T$ and let $M$ be a maximal ideal of $T$
such that $Q\subseteq M$. Then $QT_{M}$ is minimal over $P_{2}T_{M}$
which is principal by $t$-stability. By the principal ideal theorem,
$\htt(Q)=\htt(QT_{M})=1$. By the Going-Up theorem, there is a
height-two prime ideal $Q_{2}$ of $T$ contracting to $P_{2}$ in $R$.
Further, there is a minimal prime ideal $Q$ of $P_{2}$ such that
$P_{2}\subseteq Q\subsetneqq Q_{2}$. Hence $Q\cap R= Q_{2}\cap
R=P_{2}$, which is absurd since the extension $R\subset T$ is INC.
Therefore $t$-$\dim(R)=1$.
\end{proof}

\begin{lemma}\label{sec:2.8}
Let $R$ be a one-dimensional Noetherian domain. If $R$ is Boole $t$-regular, then $R$ is strongly $t$-stable.
\end{lemma}

\begin{proof}
Let $I$ be a nonzero $t$-ideal of $R$. Set $T:=(I:I)$ and
$J:=I(T:I)$. Since $R$ is Boole $t$-regular, then there is $0\not=
c\in \qf(R)$ such that $(I^{2})_{t}=cI$. Then
$(T:I)=((I:I):I)=(I:I^{2})=(I:(I^{2})_{t})=(I:cI)=c^{-1}(I:I)=c^{-1}T$.
So $J=I(T:I)=c^{-1}I$. Since $J$ is a trace ideal of $T$, then
$(T:J)=(J:J)=(c^{-1}I:c^{-1}I)=(I:I)=T$. Hence $J_{v_{1}}=T$, where
$v_{1}$ is the $v$-operation with respect to $T$.  Since $R$ is
one-dimensional Noetherian domain, then so is $T$ (\cite[Theorem
93]{Kap:1974}). Now, if $J$ is a proper ideal of $T$, then $J\subseteq N$
for some maximal ideal $N$ of $T$. Hence $T=J_{v_{1}}\subseteq
N_{v_{1}}\subseteq T$ and therefore $N_{v_{1}}=T$. Since
$\dim(T)=1$, then each nonzero prime ideal of $T$ is $t$-prime and
since $T$ is Noetherian, then $t_{1}=v_{1}$. So $N=N_{v_{1}}=T$, a
contradiction. Hence $J=T$ and therefore $I=cJ=cT$ is strongly
$t$-stable, as desired.
\end{proof}

\begin{proof}[Proof of Theorem~\ref{sec:2.6}]

$(1)\Longrightarrow (2)$ Clearly $R$ is Boole $t$-regular and, by Lemma~\ref{sec:2.7}, $t$-$\dim(R)=1$.

$(2)\Longrightarrow (1)$ Let $I$ be a nonzero $t$-ideal of $R$. Set
$T:=(I:I)$ and $J:=I(T:I)$. Since $R$ is Boole $t$-regular, then
there is $0\not= c\in \qf(R)$ such that $(I^{2})_{t}=cI$. Then
$(T:I)=((I:I):I)=(I:I^{2})=(I:(I^{2})_{t})=(I:cI)=c^{-1}(I:I)=c^{-1}T$.
So $J=I(T:I)=c^{-1}I$. It suffices to show that $J=T$. Since
$T=(I:I)=(II^{-1})^{-1}$, then $T$ is a divisorial (fractional)
ideal of $R$, and since $J=c^{-1}I$, then $J$ is a divisorial
(fractional) ideal of $R$ too. Now, for each $t$-maximal ideal $M$
of $R$, since $R_{M}$ is a one-dimensional Noetherian domain which
is Boole $t$-regular, by Lemma~\ref{sec:2.8}, $R_{M}$ is strongly
$t$-stable. If $I\not\subseteq M$, then
$T_{M}=(I:I)_{M}=(IR_{M}:IR_{M})=R_{M}$ and
$J_{M}=I(T:I)_{M}=R_{M}$. Assume that $I\subseteq M$. Then $IR_{M}$
is a $t$-ideal of $R_{M}$. Since $R_{M}$ is strongly $t$-stable,
then $IR_{M}=aR_{M}$ for some nonzero $a\in I$. Hence
$T_{M}=(I:I)R_{M}=(IR_{M}:IR_{M})=R_{M}$. Then
$J_{M}=I_{M}(T_{M}:I_{M})=R_{M}=T_{M}$. Hence
$J=J_{t}=\bigcap_{M\in
\Max_{t}(R)}J_{M}=\bigcap_{M\in
\Max_{t}(R)}T_{M}=T_{t}=T$. It follows that $I=cJ=cT$ and therefore
$R$ is strongly $t$-stable.
\end{proof}

An analogue of Theorem~\ref{sec:2.6} does not hold for Clifford $t$-regularity, as shown by the next example.

\begin{example}\label{sec:2.9}
There exists a Noetherian Clifford $t$-regular domain with $t$-$\dim(R)=1$ such that $R$ is not $t$-stable. Indeed, let us first recall that a domain $R$ is said to be pseudo-Dedekind if every $v$-ideal is invertible \cite{Kg:1987}.
In \cite{Sa:1961}, P. Samuel gave an example of a Noetherian UFD domain
$R$ for which $R[[X]]$ is not a UFD. In \cite{Kg:1987}, Kang noted that $R[[X]]$ is a
Noetherian Krull domain which is not pseudo-Dedekind; otherwise, $\Cl(R[[X]])=\Cl(R)=0$ forces $R[[X]]$ to be a UFD, absurd.
Moreover, $R[[X]]$ is a Clifford $t$-regular domain by \cite[Proposition 2.2]{KM2:2007} and clearly $R[[X]]$ has $t$-dimension 1 (since Krull). But for $R[[X]]$ not being a pseudo-Dedekind domain translates into the existence of a $v$-ideal of $R[[X]]$  that is not invertible, as desired.
\end{example}

We recall that a domain $R$ is called strong Mori if it satisfies
the ascending chain condition on $w$-ideals. Noetherian domains are strong Mori. Next we wish to extend
Theorem~\ref{sec:2.6} to the larger class of strong Mori domains.

\begin{thm}\label{sec:2.10}
Let $R$ be a strong Mori domain. Then the following conditions are equivalent:
\begin{enumerate}
\item $R$ is strongly $t$-stable;
\item $R$ is Boole $t$-regular and $t$-$\dim(R)=1$.
\end{enumerate}
\end{thm}

\begin{proof}
We recall first the following useful facts:

\smallskip{\bf Fact 1} (\cite[Lemma 5.11]{Kg:1987}). Let $I$ be a finitely generated ideal of
a Mori domain $R$ and $S$ a multiplicatively closed subset of $R$.
Then $(I_{S})_{v}=(I_{v})_{S}$. In particular, if $I$ is a $t$-ideal
(i.e., $v$-ideal) of $R$, then $I$ is $v$-finite, that is, $I=A_{v}$
for some finitely generated subideal $A$ of $I$. Hence
$(I_{S})_{v}=((A_{v})_{S})_{v}=((A_{S})_{v})_{v}=(A_{S})_{v}=(A_{v})_{S}=I_{S}$
and therefore $I_{S}$ is a $v$-ideal of $R_{S}$.

\smallskip{\bf Fact 2.} For each $v$-ideal $I$ of $R$ and each multiplicatively
closed subset $S$ of $R$, $(I:I)_{S}=(I_{S}:I_{S})$. Indeed, set
$I=A_{v}$ for some finitely generated subideal $A$ of $I$ and let $x\in
(I_{S}:I_{S})$. Then $xA\subseteq xA_{v}=xI\subseteq xI_{S}\subseteq
I_{S}$. Since $A$ is finitely generated, then there exists $\mu \in S$ such that
$x\mu A\subseteq I$. So $x\mu I=x\mu A_{v}\subseteq I_{v}=I$. Hence
$x\mu \in (I:I)$ and then $x\in (I:I)_{S}$. It follows that
$(I:I)_{S}=(I_{S}:I_{S})$.

\smallskip $(1)\Longrightarrow (2)$ Clearly $R$ is Boole $t$-regular. Let $M$ be
a maximal $t$-ideal of $R$. Then $R_{M}$ is a Noetherian domain
(\cite[Theorem 1.9]{FM:1999}) which is strongly $t$-stable. By
Theorem~\ref{sec:2.6}, $t$-$\dim(R_{M})=1$. Since $MR_{M}$ is a
$t$-maximal ideal of $R_{M}$ (Fact 1), then
$\htt(M)=\htt(MR_{M})=1$. Therefore $t$-$\dim(R)=1$.

$(2)\Longrightarrow (1)$ Let $I$ be a nonzero $t$-ideal of $R$. Set
$T:=(I:I)$ and $J:=I(T:I)$. Since $R$ is Boole $t$-regular, then
$(I^{2})_{t}=cI$ for some nonzero $c\in \qf(R)$. So $J=c^{-1}I$.
Since $J$ and $T$ are (fractional) $t$-ideals of $R$, to show that
$J=T$, it suffices to show it $t$-locally. Let $M$ be a $t$-maximal
ideal of $R$. Since $R_{M}$ is one-dimensional Noetherian domain
which is Boole $t$-regular, by Theorem~\ref{sec:2.6}, $R_{M}$ is
strongly $t$-stable. By Fact 1, $I_{M}$ is a $t$-ideal of $R_{M}$.
So $I_{M}=a(I_{M}:I_{M})$. Now, by Fact 2,
$T_{M}=(I:I)_{M}=(I_{M}:I_{M})$ and then $I_{M}=aT_{M}$. Hence
$J_{M}=I_{M}(T_{M}:I_{M})=T_{M}$, as desired.
\end{proof}

We close the paper with the following discussion about the limits as well as possible extensions of the above results.

\begin{remark}\label{sec:2.11}
\begin{xenumerate}
\item Unlike Clifford regularity, Clifford (or even Boole) $t$-regularity does not force a strong Mori domain to be Noetherian. Indeed, it suffices to consider a UFD domain which is not Noetherian.

\item Example~\ref{sec:2.4} provides a Noetherian Boole $t$-regular domain of $t$-dimension two. We do not know whether the assumption ``$t$-$\dim(R)=1$'' in Theorem~\ref{sec:2.2} can be omitted.

\item Following \cite[Proposition 2.3]{KM:2003}, the complete integral closure $\overline{R}$ of a Noetherian Boole regular domain $R$ is a PID. We do not know if $\overline{R}$ is a UFD in the case of Boole $t$-regularity. However, it's the case if the conductor $(R:\overline{R})\not=0$. Indeed, it's clear that $\overline{R}$ is a Krull domain. But $(R:\overline{R})\not =0$ forces $\overline{R}$ to be Boole $t$-regular, when $R$ is Boole $t$-regular, and by \cite[Proposition 2.2]{KM2:2007}, $\overline{R}$ is a UFD.

\item The Noetherian domain provided in Example~\ref{sec:2.4} is not strongly $t$-discrete since its maximal ideal is $t$-idempotent. We do not know if the assumption ``$R$ strongly $t$-discrete, i.e., $R$ has no $t$-idempotent $t$-prime ideals'' forces a Clifford $t$-regular Noetherian domain to be of $t$-dimension one.
\end{xenumerate}
\end{remark}

\end{section}


\end{document}